\documentclass[11pt]{article}
\pdfoutput=1

\usepackage[margin=1.3in]{geometry}

\usepackage{amssymb,amsmath}
\usepackage{amsthm}
\usepackage{hyperref,xcolor}
\usepackage{mathrsfs}
\usepackage{textcomp}
\hypersetup{
    colorlinks,%
    citecolor=black,%
    filecolor=black,%
    linkcolor=black,%
    urlcolor=black
}

\makeatletter

\newdimen\bibspace
\renewenvironment{thebibliography}[1]{%
 \section*{\refname 
       \@mkboth{\MakeUppercase\refname}{\MakeUppercase\refname}}%
     \list{\@biblabel{\@arabic\c@enumiv}}%
          {\settowidth\labelwidth{\@biblabel{#1}}%
           \leftmargin\labelwidth
           \advance\leftmargin\labelsep
           \itemsep\bibspace
           \parsep\z@skip     %
           \@openbib@code
           \usecounter{enumiv}%
           \let\p@enumiv\@empty
           \renewcommand\theenumiv{\@arabic\c@enumiv}}%
     \sloppy\clubpenalty4000\widowpenalty4000%
     \sfcode`\.\@m}
    {\def\@noitemerr
      {\@latex@warning{Empty `thebibliography' environment}}%
     \endlist}

\makeatother

\numberwithin{equation}{section}

\newtheorem{thm}{Theorem}[section]

\newtheorem{rem}[thm]{Remark}

\newcommand{\al}{\alpha}
 \newcommand{\lda}{\lambda}
\newcommand{\ud}{\mathrm{d}}

\newcommand{\om}{\Omega}
\newcommand{\R}{\mathbb{R}}

\newcommand{\be}{\begin{equation}}      \newcommand{\ee}{\end{equation}}
\newcommand{\ul}{\underline}
\newcommand{\ol}{\overline}

\begin{document}

\title{\textbf{Existence of entire solutions of Monge-Amp\`ere equations with prescribed asymptotic behaviors} \bigskip}

\author{Jiguang Bao, \ \ Jingang Xiong, \ \ Ziwei Zhou}
\date{}

\maketitle

\begin{abstract} We prove the existence of entire solutions of the Monge-Amp\`ere equations with prescribed asymptotic behavior at infinity of the plane, which was left by Caffarelli-Li in 2003. The special difficulty of the problem in dimension two is due to the \textit{global logarithmic term} in the asymptotic expansion of solutions at infinity. Furthermore, we give a PDE proof of the characterization of the space of solutions of the Monge-Amp\`ere equation $\det \nabla^2 u=1$  with $k\ge 2$ singular points, which was established by G\'alvez-Mart\'inez-Mira in 2005. We also obtain the existence in higher dimensional cases with general right hand sides. \\[1mm]
\end{abstract}

\section{Introduction}

In 1954, K. J\"orgens \cite{J1} proved that, modulo the unimodular affine equivalence, $\frac{1}{2}|x|^2$ is the unique convex smooth solution of
\[
\det \nabla^2 u=1 \quad \mbox{in }\R^2.
 \]
J\"orgens theorem was extended to smooth convex solutions in higher dimensions by Calabi \cite{Cal} for less than or equal to 5 dimensions and by Pogorelov \cite{P} for all dimensions. Different proofs were given by  Cheng-Yau \cite{CY}, Caffarelli \cite{Ca} and Jost-Xin \cite{JoXi}. In dimension two, elementary and simpler proofs were found by Nitsche \cite{N} and Jin-Xiong \cite{JX1}.

In \cite{CL}, Caffarelli and Li established a quantitative version of the theorem of J\"orgens-Calabi-Pogorelov. They considered
\begin{equation} \label{eq:ma-1}
\det \nabla^2 u=f \quad \mbox{in }\R^n,
\end{equation}
where $f\in C^0(\R^n)$ satisfies that
\begin{equation} \label{eq:c-1}
\begin{split}
0<\inf_{\R^n} f\le \sup_{\R^n} f<\infty~\mbox{and supp}(f-1) \mbox{ is bounded}.
\end{split}
\end{equation}

Denote
$$\mathcal{A}:=\{A:~A~\text{is a symmetric, positive definite $n\times n$ matrix and}~\det{A}=1\}.$$

\begin{thm}[Caffarelli-Li \cite{CL}]  Let $u$ be a convex viscosity (Alexsandrov) solution of \eqref{eq:ma-1} with $f$ satisfying \eqref{eq:c-1}. Then $u\in C^\infty(\R^n\setminus supp(f-1))$, and we have the following:
\begin{itemize}
\item[-] For $n\ge 3$, there exist a linear function $\ell(x)$ and $A\in \mathcal{A}$ such that
\be \label{eq:c-2}
\limsup_{x\to \infty }|x|^{n-2}\big|u(x)-(\frac{1}{2} x^t Ax+\ell(x))\big|<\infty.
\ee
\item[-] For $n=2$, there exist a linear function $\ell(x)$ and $A\in \mathcal{A}$ such that
\be \label{eq:c-3}
\limsup_{x\to \infty }|x|\big|u(x)-(\frac{1}{2} x^t Ax+d\ln\sqrt{x^t Ax}+\ell(x))\big|<\infty,
\ee
where
\be \label{eq:c-4}
d= \frac{1}{2\pi} \int_{\R^2}(f-1)\,\ud x.
\ee
\end{itemize}
\end{thm}

The asymptotic behaviors in exterior domains of dimension two had been established by Ferrer-Mart\'inez-Mil\'an \cite{FMM}.

In addition, Caffarelli-Li \cite{CL} proved that \eqref{eq:ma-1} with the condition \eqref{eq:c-2} admits a unique viscosity solution when $n\ge 3$; see Theorem 1.7 of \cite{CL}.  However, it was not known whether \eqref{eq:ma-1} with the condition \eqref{eq:c-3} has a unique solutions in the plane. The difficulty stems from the global constant $d$ in \eqref{eq:c-3}, which makes it hard to construct sub- and supper- solutions with quadratic growth.   In this paper,  we answer the problem positively by a different method.

In fact, we can relax the assumption on $f$. Let $\nu$ be a locally finite Borel measure defined in $\R^2$ and $\ud \nu=f\,\ud x$ in $\R^n\setminus \om$, where $\om$ is a bounded open set and $f\in C^3(\R^n\setminus \om)$ is positive function satisfying
\be \label{eq:c-5}
\limsup_{|x|\to\infty}|x|^{\beta+j}|\nabla^j(f(x)-1)|<\infty, \quad j=0,1,2,3,
\ee
for some $\beta>2$. By Corollary 1.1 in \cite{BLZ} that for every Alexsandrov solution of
\be \label{eq:ma-2}
\det \nabla^2 u=\nu \quad \mbox{in }\R^n,
\ee
for $n\ge3$, there exist a linear function $\ell(x)$ and $A\in \mathcal{A}$ such that
\be \label{eq:c-3''}
\limsup_{x\to \infty }|x|^{\min\{\beta,n\}-2+j}\big|\nabla^j(u(x)-(\frac{1}{2} x^t Ax+\ell(x)))\big|<\infty,
\ee
where $j=0,1,2,3,4$; for $n=2$, there exist a linear function $\ell(x)$ and $A\in \mathcal{A}$ such that
\be \label{eq:c-3'}
\limsup_{x\to \infty }|x|^{\sigma+j}\big|\nabla^j(u(x)-(\frac{1}{2} x^t Ax+d\ln\sqrt{x^t Ax}+\ell(x)))\big|<\infty,
\ee
where $j=0,1,2,3,4$, $\sigma\in (0,\min\{\beta-2,2\})$, and similar to the proof of (1.9) in \cite{CL},
\be \label{eq:c-6}
d= \frac{1}{2\pi} \lim_{R\to \infty} \Big(\int_{B_R}\,\ud \nu -\pi R^2\Big).
\ee

The main result of this paper is
\begin{thm} \label{thm:main} Let $n=2$ and $\nu$ be as above. For any linear function $\ell$ and $A\in \mathcal{A}$,  the Monge-Amp\`ere equation  \eqref{eq:ma-2}
 has a unique Alexsandrov solution satisfying \eqref{eq:c-3'} with $d$ given by \eqref{eq:c-6}.

\end{thm}

Theorem \ref{thm:main} confirms the Conjecture 1 of \cite{BLZ} particularly. Our proof is very different from the one in \cite{CL} for $n\ge 3$, where sub- and supper- solutions are constructed. By the method of \cite{CL}, we have the following theorem.

\begin{thm}\label{meah} Let $n\ge 3$ and $\nu$ be as above. For any linear function $\ell$ and $A\in\mathcal A$, the Monge-Amp\`ere equation \eqref{eq:ma-2} has a unique Alexandrov solution satisfying \eqref{eq:c-3''}.
\end{thm}

\begin{rem}
The condition $\beta>2$ is necessary for the asymptotic behavior \eqref{eq:c-3''} and \eqref{eq:c-3'}. Let $f$ be a radial, smooth, positive function satisfying $f(r)\equiv1$ for $r\in[0,1]$ and $f(r)=1+r^{-2}$ for $r>2$. Then
$$u(x)=\int_0^{|x|}\Big(\int_0^s2tf(t)\,\ud t\Big)^{\frac12}\,\ud s$$
is a solution of \eqref{eq:ma-2} with $\ud \nu=f\,\ud x$ in $\R^n$.  But, as $|x|\to \infty$,
\begin{align*}
u(x)=\begin{cases}
\frac12{|x|}^2+O((\log|x|)^2)& \quad \mbox{for }n=2,\\[2mm]
\frac12|x|^2+O(\log|x|) & \quad \mbox{for }n\ge 3.
\end{cases}
\end{align*}

\end{rem}

In 1955, J\"orgens \cite{J2} further proved that, modulo the unimodular affine equivalence, every smooth locally convex solution of
\[
\det \nabla^2 u=1 \quad \mbox{in }\R^2\setminus \{0\}
 \]
 has to be
 \be \label{eq:uniq}
 \int_0^{|x|}(c+t^2)^{1/2}\,\ud t, \quad c\ge 0.
 \ee
In 2016, Jin-Xiong \cite{JX2} extended J\"orgens theorem to all dimensions, i.e,  \[
\int_0^{|x|}(c+t^n)^{1/n}\,\ud t, \quad c\ge 0
\] is the unique solution of $\det \nabla^2 u=1$ in $\R^n\setminus \{0\}$, $n\ge 3$,   upon the unimodular affine equivalence.
Furthermore, they identified the set of local convex entire solutions with $k\ge 1$ singular points to an orbifold of dimension
$d(n,k)$, where
  \[
d(n,k)=
\begin{cases}
k-1+\frac{(k-1)k}{2},\quad\mbox{if}\quad k-1\le n,\\
k-1+\frac{n(n+1)}{2}+(k-1-n)n,\quad\mbox{if}\quad k-1> n
\end{cases}
\]
when $n\ge 3$.
The later result in dimension two was obtained by G\'alvez-Mart\'inez-Mira \cite{GMM}, using a complex analysis method.
Jin-Xiong's proof is based on the result which they proved: If $u$ is a locally convex solution of
$$\det \nabla^2u=1~~{\rm{in}}~\R^n\backslash\{P_1,\cdots,P_k\},$$
then there exist nonnegative  constants $c_i$ such that
\[
\det \nabla^2 u=1+\sum_{i=1}^k c_i\delta_{P_i}
\]
in the Alexsandrov sense, where $P_i$, $i=1,\cdots, k$, are distinct points, and $\delta_{P_i}$ is the Dirac measure centered at $P_i$. This result holds for all $n\ge 2$. Together with the asymptotic behavior at infinity, we have all the parameters to determinate the dimensions of the orbifolds. It remains to show existence.  \cite{JX2} proved existence when $n\ge 3$. Theorem \ref{thm:main} applies here to obtain existence in dimension two.



Finally, we would like to mention a further extension of the theorem of J\"orgens-Calabi-Pogorelov. In another paper \cite{CL-1}, Caffarelli-Li classified entire solutions of Monge-Amp\`ere equations with periodic functions on the right hand side. See also the recent work of Teixeira-Zhang \cite{TZ}.

The paper is organized as follows. Theorem \ref{thm:main} is proved in the next section. Using the arguments of \cite{CL} and \cite{JX2}, we give a proof of Theorem \ref{meah} in section 3.

\medskip

\noindent{\bf Acknowledgement:} All authors are supported in part by the key project NSFC 11631002, and J.Xiong is also supported in part by NSFC 11501034 and 11571019.

\section{Proof of Theorem \ref{thm:main}}

For convenience, we recall the definition of Alexsandrov solutions, see e.g., Gutierrez \cite{G} and Figalli \cite{F}.
Let $\Omega$ be an open subset of $\R^n$ and $u:~\Omega\to\R$ be a locally convex function. The normal mapping of $u$, or subdifferential of $u$, at $x_0\in\Omega$ is the set-valued function $\partial u:~\Omega\to\mathcal{P}(\R^n)$ defined by
$$\partial u(x_0)=\{p:~u(x)\ge u(x_0)+p\cdot(x-x_0),~{\rm{for~all}}~x\in\Omega\},$$
where $\mathcal{P}(\R^n)$ denotes the class of all subsets of $\R^n$. Given $E\subset\Omega$, define $\partial u(E)=\cup_{x\in E}\partial u(x)$. One can show that the class
$$S=\{E\subset\Omega:\partial u(E)~{\text{is Lebesgue measurable}}\}$$
is a Borel $\sigma$-algebra.
The set function $Mu:S\to\overline{\R}$ defined by
$$Mu(E)=|\partial u(E)|$$
 is called the Monge-Amp\`ere measure associated with the function $u$, where $|\cdot|$ is the $n$-dimensional Lebesgue measure.  For a Borel measure $\nu$ in $\Omega$, we say a locally convex function $u$ is an Alexsandrov solution of the Monge-Amp\`ere equation
$$\det\nabla^2u=\nu$$
if the Monge-Amp\`ere measure $Mu$ equals $\nu$.

Now we start to prove Theorem \ref{thm:main}.
\begin{proof}[Proof of Theorem \ref{thm:main}] We only need to prove the existence part as the uniqueness part follows from the comparison principle. By the affine invariance, we can assume that $A$ is the identity matrix $I$ and $\ell=0$.

Take $\rho>0$ such that $\om \subset B_\rho$. Let
\begin{align*}
\underline{f}(r)=\begin{cases}
0,& \quad r<\rho,\\
 \min_{x\in \partial B_r} f(x) , & \quad r\ge \rho,
 \end{cases}
\end{align*}
$$\quad \underline d=  \frac{1}{2\pi}\int_{\R^2} (\underline f-1)\,\ud x=\int_0^\infty r(\underline f(r)-1)\,\ud r,$$
and
\[
w_c(r)= \int_{0}^r \Big(\int_0^s 2 t \underline f(t)\,\ud t+2c\Big)^{1/2}\,\ud s,
\]
where $c\ge 0$. It is easy to check that $w_c$ is a convex solution of
\be\label{eq:w-c}
\det \nabla^2w_c=|\partial w_c(0)|\delta_0+\underline{f}=2\pi c\delta_0+\underline{f} \quad \mbox{in }\R^2.
\ee
Using the condition \eqref{eq:c-5} on $f$, by a direct calculation we have
\begin{align*}
&\Big(\int_0^s 2 t \underline f(t)\,\ud t+2c\Big)^{1/2}\\
=&\Big(s^2+2c+\int_0^{\infty} 2 t (\underline f(t)-1)\,\ud t-\int_s^{\infty} 2 t (\underline f(t)-1)\,\ud t\Big)^{1/2}\\
=&\Big(s^2+2(c+\underline{d})+O(s^{2-\beta})\Big)^{1/2}\\
=&s\Big(1+2(\underline{d}+c)s^{-2}+O(s^{-\beta})\Big)^{1/2}\\
=&s\Big(1+(\underline{d}+c)s^{-2}+O(s^{-\min\{\beta,4\}})\Big)\\
=&s+(\underline{d}+c)s^{-1}+O(s^{-\min\{\beta-1,3\}}) \quad \mbox{as }s\to\infty.
\end{align*}
Thus,
\begin{align*}
h(s):=\Big(\int_0^s 2 t \underline f(t)\,\ud t+2c\Big)^{1/2}-s-(\underline{d}+c)s^{-1}=O(s^{-\min\{\beta-1,3\}})
\end{align*}
as $s\to\infty$. It follows that
\begin{align}\label{eq:ab}
w_c(r)=& \int_{1}^r \Big(\int_0^s 2 t \underline f(t)\,\ud t+2c\Big)^{1/2}\,\ud s+\int_{0}^1 \Big(\int_0^s 2 t \underline f(t)\,\ud t+2c\Big)^{1/2}\,\ud s\nonumber \\
=&\frac12r^2+(\underline{d}+c)\ln r-\frac12+\int_{1}^{\infty} h(s)\,\ud s-\int_{r}^{\infty} h(s)\,\ud s \nonumber\\
&+\int_{0}^1 \Big(\int_0^s 2 t \underline f(t)\,\ud t+2c\Big)^{1/2}\,\ud s \nonumber \\
=&\frac12r^2+(\underline{d}+c)\ln r+O(1) \quad \mbox{as }r\to\infty.
\end{align}

Let
$$\bar c=d-\ul d.$$
By \eqref{eq:w-c},
\[
\det \nabla^2 w_{\bar c}= 2\pi \bar c \delta_0 + \underline f \quad \mbox{in }\R^2.
\]
Since 
$$d=\frac{1}{2\pi}(\int_\Omega\,\ud \nu-\int_\Omega\,\ud x+\int_{\R^2\backslash\Omega}(f-1)\,\ud x),$$
and
$$\underline d=\frac{1}{2\pi}(-\int_\Omega\,\ud x+\int_{\R^2\backslash\Omega}(\underline f-1)\,\ud x),$$
we have
\begin{align*}
\bar c&= \frac{1}{2\pi}(\int_\Omega\,\ud \nu+\int_{\R^2\backslash\Omega}(f-\underline{f})\,\ud x).
\end{align*}
For any large $R>\rho$, choose $\lda_{\bar c}(R)$ such that
\[
w_{\bar c}(R) +\lda_{\bar c}(R)= \frac{R^2}{2}+ d \ln R.
\]
By \eqref{eq:ab}, $\lda_{\bar c}(R)$ is uniformly bounded in $R>\rho$.

Let $u_R\in C(\bar B_R)$ be the unique Alexandrov solution of
\begin{align}\label{di}
\begin{cases}
\det \nabla^2 u_R=\nu \quad & \mbox{in }B_R,\\
u_R=\frac{R^2}{2}+d \ln R \quad & \mbox{on }\partial B_R;
\end{cases}
\end{align}
see Theorem 1.6.2 in \cite{G}.

We claim that $ u_R(0)\ge\lda_{\bar c}(R)$.

Indeed, for any large $R$ and any $c>\bar c$, let $\lda_c(R)\in \R$ such that
\[
w_{ c}(R) +\lda_c(R)= \frac{R^2}{2}+ d \ln R.
\]
If $u_R(0)\le\lda_c(R)$, then, considering that for any Borel set $E\subset B_R\backslash\{0\}$,
$$|\partial u_R(E)|=\int_E\,\ud \nu\ge\int_E\underline{f}\,\ud x=|\partial(w_c+\lda_c(R))(E)|,$$
it follows from the comparison principle that $u_R(x) \le w_{ c}(| x|) +\lda_c(R)$ for all  $x\in B_R$. By Lemma 1.4.1 in \cite{G}, we have $\partial (w_{c} +\lda_c(R))(B_R)\subset \partial u_R(B_R)$. However, note that when $c>\bar c$,
\begin{align*}
|\partial (w_{c} +\lda_c(R))(B_R)|&=\int_{B_R}\underline{f}\,\ud x+2\pi c\\
&>\int_\Omega\,\ud \nu+\int_{B_R\backslash\Omega}f\,\ud x+\int_{\R^2\backslash B_R}(f-\ul f)\,\ud x\\
&\ge\int_{B_R}\,\ud \nu=| \partial u_R(B_R)|.
\end{align*}
Hence, we have derived a contradiction. It follows that $u_R(0)>\lda_c(R)$. For any fixed $R$, since $w_c(R)$ is continuous with respect to $c$,  $\lda_c(R)$ is continuous with respect to $c$. Sending $c\to \bar c$, we have
\begin{align} \label{eq:b}
u_R(0)\ge\lda_{\bar c}(R)
\end{align}
for any large $R$.

Let $v_R(x)= w_{\bar c}(|x|) +u_R(0)$. By \eqref{eq:b}, we have
$$v_R \ge w_{\bar c}(R) +\lda_{\bar c}(R)= u_R \quad \mbox{on }\partial B_R,$$
and $v_R(0)= u_R(0)$.  By the comparison principle, we have
\begin{align} \label{eq:a}
v_R\ge u_R \quad \mbox{in }B_R.
\end{align}

Since $u_R$ is a convex function, there exists a vector $p_R(0)$ such that
\[
 u_R(x)\ge p_R(0)x +u_R(0) \quad \mbox{for all }x\in B_R.
\]
By \eqref{eq:a}, we have
\begin{align*}
p_R(0)x\le w_{\bar c}(|x|)+u_R(0)-u_R(0)\le w_{\bar c}(|x|).
\end{align*}
It follows that $|p_R(0)|\le C$ for some constant $C$ independent of $R$.

Let $\tilde{u}_R(x)= u_R(x)- (p_R(0)x +u_R(0))$. Note that
\begin{align*}
0&\le\tilde{u}_R(x)\le v_{R}(x) - (p_{R}(0)x +u_{R}(0))\\
&=w_{\bar c}(|x|)+ u_R(0)- p_{R}(0)x-u_{R}(0)\\
&\le w_{\bar c}(|x|)+C|x|
\end{align*}
and
\[
\det \nabla^2  \tilde{u}_R  = \nu \quad \mbox{in }B_{R}.
\]
By the Lipschitz estimates for convex function (see, e.g.,  Theorem 6.7 in \cite{EG}), for any $K\subset\subset B_{R/2}$,
$$||\tilde{u}_R||_{C^{0,1}(K)}\le C(K),$$
where $C(K)$ is a constant independent of $R$. Then after passing to subsequence, denoted by $\tilde{u}_{R_i}$,  we have
\[
\tilde{u}_{R_i} \to u_\infty \quad \mbox{in }C_{loc}^\al(\R^2)
\]
where $\alpha\in (0,1)$ for some convex function $u_\infty$ satisfying
\begin{align} \label{eq:2}
0\le u_\infty(x) \le w_{\bar c}(|x|)+C|x|= \frac{1}{2}|x|^2+C|x|+d\ln |x|+O(1)
\end{align}
 and
\[
\det \nabla^2  u_\infty =\nu \quad \mbox{in } \R^2
\]
in the Alexandrov sense. It follows from Corollary 1.1 in \cite{BLZ} that there exist $A\in\mathcal{A}$, $D\in\R$ and a linear function $\ell(x)$ such that
\be\label{eq:lll}
\limsup_{|x|\to\infty}|x|^{j+\sigma}|\nabla^j(u_\infty(x)-(\frac12x'Ax+D\ln\sqrt{x'Ax}+\ell(x)))|<\infty,
\ee
for $j=0,1,2,3,4$, and $\sigma\in(0,\min\{\beta-2,2\})$. By \eqref{eq:2}, $A$ can not have one eigenvalue greater than $1$. This forces that all the eigenvalues equal $1$ and thus $A=I$.

To prove $D=d$, we use the method of proving (1.9) in \cite{CL}. Let $u=u_\infty-\ell$. We first assume that $u\in C^3(\R^2)$ and write
$$E(x)=u(x)-(\frac12|x|^2+D\ln|x|),$$
and
$$\det\nabla^2u=\partial_1(u_1u_{22})-\partial_2(u_1u_{12}).$$
By \eqref{eq:lll}, as $|x|\to\infty$,
$$|E(x)|=O(|x|^{-\sigma}),~|DE(x)|=O(|x|^{-\sigma-1}),~|D^2E(x)|=O(|x|^{-\sigma-2}).$$
Integrating the equation of $u$ on $B_R$ and integrating by parts, we have, as $R\to\infty$,
\begin{align*}
\int_{B_R}\,\ud\nu&=\int_{B_R}\partial_1(u_1u_{22})-\partial_2(u_1u_{12})\,\ud x\\
&=\int_{|x|=R}\Big[u_1u_{22}\frac{x_1}{|x|}-u_1u_{12}\frac{x_2}{|x|}\Big]\,\ud x\\
&=\int_{|x|=R}\Big[(x_1+\frac{Dx_1}{|x|^2}+E_1)(1+D\frac{|x|^2-2x_2^2}{|x|^4}+E_{22})\frac{x_1}{|x|}\\
&~~~~~~~~~~~~~~~~-(x_1+\frac{Dx_1}{|x|^2}+E_1)(-2D\frac{x_1x_2}{|x|^4}+E_{12})\frac{x_2}{|x|}\Big]\,\ud x\\
&=\int_{|x|=R}(x_1+\frac{Dx_1}{|x|^2})(\frac{x_1}{|x|}+\frac{Dx_1}{|x|^3})\,\ud x+O(R^{-\sigma})\\
&=\int_{|x|=R}(\frac{x_1^2}{|x|}+\frac{2Dx_1^2}{|x|^3})\,\ud x+O(R^{-\sigma})\\
&=\pi R^2+2\pi D+O(R^{-\sigma}),
\end{align*}
where  $E_i=\partial_i E$ and $E_{ij}=\partial_{ij}^2E$ for $i,j=1,2$. 
Sending $R$ to infinity, we have $D=d$.

For $u\in C(\R^2)$, by \eqref{eq:lll} we know that $u$ is of $C^4$ near $\partial B_R$ for large $R$. Let $u_\epsilon\in C^\infty(\R^2)$ be a family of convex function such that $u_\epsilon\to u$ in $C_{loc}^0(\R^2)$, and $u_\epsilon\to u$ in $C^4$ near $\partial B_R$ as $\epsilon\to 0$. Let $\eta$ be a continuous cutoff function satisfying $\eta=1$ in $B_R$, and $\eta=0$ in $\R^2\backslash B_{R+1}$. By Lemma 1.2.3 in \cite{G},
$$\lim_{\epsilon\to0}\int_{\R^2}\eta\det\nabla^2u_\epsilon\,\ud x=\int_{\R^2}\eta\,\ud\nu.$$
Note that
$$\lim_{\epsilon\to0}\int_{B_{R+1}\backslash B_R}\eta\det\nabla^2u_\epsilon\,\ud x=\int_{B_{R+1}\backslash B_R}\eta\,\ud\nu.$$
Subtracting the two equalities above, we have
$$\lim_{\epsilon\to0}\int_{B_R}\det\nabla^2u_\epsilon\,\ud x=\int_{B_R}\,\ud\nu.$$
As shown above,
$$\int_{B_R}\det\nabla^2u_\epsilon\,\ud x=\int_{|x|=R}\Big[u_{\epsilon1}u_{\epsilon22}\frac{x_1}{|x|}-u_{\epsilon1}u_{\epsilon12}\frac{x_2}{|x|}\Big]\,\ud x.$$
Sending $\epsilon\to0$, we have
\begin{align*}\int_{B_R}\,\ud\nu&=\int_{|x|=R}\Big[u_1u_{22}\frac{x_1}{|x|}-u_1u_{12}\frac{x_2}{|x|}\Big]\,\ud x\\&
=\pi R^2+2\pi D+O(R^{-\sigma}).
\end{align*}
Sending $R$ to infinity, again we have $D=d$. Then $u$ is the solution we want.

Therefore, Theorem \ref{thm:main} is proved.
\end{proof}

\section{Proof of Theorem \ref{meah}}
\label{sec:3}

When $f\equiv1$ outside $\om$, Theorem \ref{meah} was proved by \cite{JX2}.

We only show the existence part as the uniqueness part follows from the comparison principle. Due to the affine invariance, we  assume that $A=I$, $\ell=0$ and $\om\subset B_{\frac12}$.  We assume $ \ud \nu=f\ud x$ in $\R^n$ and $f\in C^\infty(\R^n)$ is positive and satisfies \eqref{eq:c-5}. The bounds we will obtain are independent of the smoothness and the lower bound of $f$ in $B_{1/2}$. By an approximation argument, Theorem \ref{meah} will follow.

Next we are going to construct sub- and super- solutions by following the arguments in \cite{CL} and \cite{JX2}.

Let $\eta$ be a nonnegative smooth function supported in $B_{\frac14}$ satisfying $\int_{B_1}\eta\,\ud x=1$, and $v_1$ be the smooth solution of
$$\begin{cases}
\det\nabla^2 v_1=f+a\eta& \quad \mbox{in }B_1,\\
v_1=0& \quad \mbox{on } \partial B_1,
\end{cases}$$
where $a>0$ will be chosen later. It follows from Alexandrov's maximum principle (see, e.g., Theorem 1.4.2 in \cite{G}) that
$$v_1\ge -c(n)|\partial v_1(B_1)|^{\frac 1n}=-c(n)\Big(\int_{B_1}f(x)\,\ud x+a\Big)^{\frac 1n}=:-c_0 \quad \mbox{in }B_{\frac12},$$
where $c(n)$ is a constant depending only on the dimension $n$.

Let $r=|x|$ and define
$$\bar f(r)=\max_{|x|=r}f(x),\quad r\ge\frac12.$$
Let $c_1=\int_{\frac12}^1(\int_1^snt^{n-1}\bar f(t)\,\ud t)^{\frac 1n}\,\ud s$, $K=\frac{c_0}{c_1}$,
$$v_2(r)=
\begin{cases}
K\int_1^r(\int_1^snt^{n-1}\bar f(t)\,\ud t)^{\frac 1n}\,\ud s,& \quad r\ge\frac12,\\
-c_0,& \quad 0\le r<\frac12.
\end{cases}$$
First of all, $v_1\ge v_2$ in $\bar B_{\frac12}$. Secondly, by  choosing $a$ large such that $c_0\ge c_1$, we have
$$\det \nabla^2v_2={K}^n\bar f\ge f=\det\nabla^2 v_1 \quad \mbox{in } B_1\setminus \bar B_{\frac12} ,$$
and $v_1=v_2=0$ on $\partial B_1$. By the comparison principle, we have  $v_1\ge v_2$ in $B_1\setminus \bar B_{\frac12} $. So $v_1\ge v_2$ in $B_1$.

Let
$$\underline u(x)=
\begin{cases}
\int_1^r(\int_1^snt^{n-1}\bar f(t)\,\ud t+K)^{\frac 1n}\,\ud s,~~r\ge1,\\
v_1,~~~~~~~~~~~~~~~~~~~~~~~~~~~~~~~~~~~~~~~~~0\le r<1.
\end{cases}$$
Then $\underline u\in C^0(\R^n)\cap C^\infty(B_1)\cap C^\infty(\R^n\setminus \bar B_1)$, $\underline u$ is locally convex in $\R^n\backslash B_1$,
\begin{align*}
&\det\nabla^2\underline u=\bar f~~{\rm{in}}~\R^n\backslash\overline{B_1},\\
&\det\nabla^2\underline u\ge f~~{\rm{in}}~B_1.
\end{align*}
Moreover, we have $\underline u\ge v_2$ in $B_1$, and $\underline u=v_2$ on $\partial B_1$, then
$$\lim_{r\to1^-}\partial_r\underline u\le\lim_{r\to1^-}\partial_r v_2.$$
Since
$$\lim_{r\to1^-}\partial_r v_2=0<(K)^{\frac 1n}=\lim_{r\to1^+}\partial_r\underline u,$$
we have
\begin{equation}\label{eq:under}
\lim_{r\to1^-}\partial_r\underline u<\lim_{r\to1^+}\partial_r\underline u.
\end{equation}
It follows that $\underline u$ is convex in $\R^n$.
By a simple computation,
$$\sup_{\R^n}\Big|\ul u(x)-\frac12|x|^2\Big|\le C$$
for some $C>0$ depending  only on $n$, $\int_{B_1}f(x)\,\ud x$ and $\bar f$ outside $B_{1/2}$.

Define
$$\underline f(r)=\min_{|x|=r}f(x),\quad r\ge\frac12,$$
and
$$\bar u(x)=
\begin{cases}
\int_1^{|x|}(\int_1^snt^{n-1}\underline{f}(t)\,\ud s)^{\frac 1n}\,\ud s,~~|x|>1,\\
0,~~~~~~~~~~~~~~~~~~~~~~~~~~~~~~~~~~~~~~~~~~|x|\le1.
\end{cases}$$
It follows that
\be\label{eq:over}
\lim_{r\to1^-}\partial_r\bar u=\lim_{r\to1^+}\partial_r\bar u=0,
\ee
and
$$\sup_{\R^n}\Big|\bar u(x)-\frac12|x|^2\Big|<+\infty.$$

By the above construction, we have
$$\beta_+:=\sup_{\R^n}(\frac{|x|^2}{2}-\bar u(x))<+\infty~~{\rm{and}}~~\beta_-:=\inf_{\R^n}(\frac{|x|^2}{2}-\underline u(x))>-\infty$$
which depend only on  $n$, $\int_{B_1}f(x)\,\ud x$ and $f$ outside $B_{1/2}$.

For $R>1$, let $u_R$ be the unique convex smooth solution of
$$\begin{cases}
\det\nabla^2 u_R=f& \quad \mbox{in }B_R,\\
u_R=\frac{R^2}{2}&\quad \mbox{on } \partial B_R.
\end{cases}$$
We claim that
\begin{equation}\label{eq:lob}
\underline u(x)+\beta_-\le u_R(x)\le \bar u(x)+\beta_+, \quad x\in B_R.
\end{equation}

To establish the first inequality, let $\ul x$ be a maximum point of the function
$$\ul h(x):=\underline u(x)+\beta_--u_R(x)$$
in $\bar B_R $. Since
$$\det\nabla^2\ul u\ge\det\nabla^2u_R \quad \mbox{in }B_R\setminus \bar B_1$$
and
$$\det\nabla^2\ul u\ge\det\nabla^2u_R~~{\rm{in}}~B_1,$$
we have, by the strong maximum principle, $\ul x\in \partial B_R$ or $\ul x\in\partial B_1$.  If $\ul x\in\partial B_R$, then by the definition of $\beta_-$,
$$\ul h(x)\le\underline u(\ul x)+\beta_--u_R(\ul x)\le\frac{|\ul x|^2}{2}-\frac{R^2}{2}=0 \quad \mbox{in }\bar B_R$$
and the inequality holds. If $\ul x\in \partial B_1$, then considering the smoothness of $u_R$, it contradicts to the condition \eqref{eq:under}. Hence, the first inequality of \eqref{eq:lob} holds. For the second inequality, let $\bar x$ be a minimum point of the function
$$\bar h(x):=\bar u(x)+\beta_+-u_R(x)$$
in $\ol{B_R}$. Similar to the above, $\bar x\in\partial B_R$ or $\bar x \in\partial B_1$. If $\bar x\in\partial B_R$, then by the definition of $\beta_+$,
$$\bar h(x)\ge\bar u(\bar x)+\beta_+-u_R(\bar x)\ge\frac{|\bar x|^2}{2}-\frac{R^2}{2}=0 \quad \mbox{in }\bar B_R $$
and the inequality holds. If $\bar x\in\partial B_1$, in view of \eqref{eq:over} and the equation $u_R$ satisfies, this is impossible. Then the inequality \eqref{eq:lob} holds.

By \eqref{eq:lob} and the Lipschitz estimate for convex functions (see Theorem 6.7 in \cite{EG}), we have, along a subsequence $R_i\to\infty$,
$$u_{R_i}\to u_\infty \quad \mbox{in }C^\al_{loc}(\R^n),$$
where $0<\al<1$, $u_\infty$ satisfies $\det \nabla^2 u_\infty=f $ in $\R^n$ in the Alexandrov sense and
\[
\underline u(x)+\beta_-\le u_\infty (x)\le \bar u(x)+\beta_+ \quad \mbox{in }\R^n,
\]
which particularly implies that
\be\label{eq:in}
\sup_{\R^n}\Big|u_\infty(x)-\frac12|x|^2\Big|\le C
\ee
for some $C>0$ depending  only on $n$, $\int_{B_1}f(x)\,\ud x$ and $f$ outside $B_{1/2}$. By Bao-Li-Zhang \cite{BLZ}, there exist $A\in\mathcal{A}$ and a linear function $\ell(x)$ such that \eqref{eq:c-3''} holds for $j=0,1,2,3,4$. Considering \eqref{eq:in}, we have $A=I$ and $\ell=\tilde{c}$ for some constant $\tilde{c}$. Then
\[
u=u_\infty-\tilde{c}
\]
is the solution we want.

Therefore, we complete the proof of Theorem \ref{meah}.

\bigskip

\noindent J. Bao, J. Xiong \& Z. Zhou

\medskip

\noindent  School of Mathematical Sciences, Beijing Normal University\\
Laboratory of Mathematics and Complex Systems, Ministry of Education\\
Beijing 100875, China \\[1mm]
Email: \textsf{jgbao@bnu.edu.cn, jx@bnu.edu.cn, zhouziwei@mail.bnu.edu.cn}

\end{document}